\documentclass[11pt,reqno]{amsart}
\usepackage{amssymb,verbatim}
\usepackage[english]{babel}
\usepackage{hyperref}

\let\le\leqslant  \let\phi\varphi

\let\opn\operatorname 

\theoremstyle{definition}
\newtheorem{DEF}{Definition}
\newtheorem*{EXAM*}{Example}

\theoremstyle{remark}
\newtheorem*{REM*}{Remark} 

\theoremstyle{plain}
\newtheorem{THM}{Theorem}
\newtheorem{PROP}{Proposition}
\newtheorem*{COR*}{Corollary}   

\address{A.~M.~Vershik, St.~Petersburg Department of Steklov Institute
of Mathematics, 27 Fontanka, St.~Petersburg
191023, Russia.}
\email{vershik@pdmi.ras.ru}
\address{M.~I.~Graev, Institute for System Studies, 36-1 Nakhimovsky
pr., 117218 Moscow, Russia.}
\email{graev\_36@mtu-net.ru}

\title{}


\begin{document}


\maketitle

\markboth{A.~M.~Vershik and M.~I.~Graev}
{Special representations of nilpotent Lie groups}

\begin{center}{}{\bf
Special representations of nilpotent Lie groups
and the associated Poisson representations of
current groups} \end{center}
\bigskip
\begin{center}
A.~M.~Vershik \footnote{Partially supported by the RFBR grants
11-01-00677-a, 11-01-12092-ofi-m, 10-01-90411-ukr-a.} and
M.~I.~Graev \footnote{Partially supported by the RFBR grant
10-01-00041a.} 
 \end{center}
\bigskip
\rightline{\it \textbf{To the memory of our teacher and coauthor, the great Israel Gelfand}}

\vskip-1cm

{\tableofcontents}

The goal of this paper is to give a short version of the new model of the representation
of the current groups with a semisimple Lie group like  $(U(n,1))^X, (O(n,1))^X$. In the earlier
papers of 70-80-th with I.M.Gelfand who had posed the problem about irreducible representations of
the current group for $SL(2,R)$, we have used a Fock space (\cite{VGG,GGV}) similar to Araki's formalism
(\cite{Ar}).  The construction of the Fock unitary representation of the current group
$G^X$, where $G$ is an arbitrary locally compact group, is based on a special
representation of the original group~$G$, i.e., a unitary
representation of $G$ with nontrivial 1-cohomology (see, e.g.,
\cite{Ar, VGG}). Some groups, including a part of semisimple rank~$1$
Lie groups, have irreducible special unitary representations; see
\cite{VG, GGV} and the references in \cite{VG}. A necessary condition
for this is that the identity representation is not isolated
in the space of irreducible representations (i.e., the absence of
Kazhdan's property), see \cite{Karp, Del}. The representations of the
current groups $G^X$ corresponding to irreducible special
representations are also irreducible.

However after our papers (see \cite{VG,VG-1} and previous) became clear that
the crucial role in the constructions play nilpotent subgroups and
its one-dimensional extensions. In the same time for commutative and nilpotent
groups $G$, only the identity representation is special, and the
corresponding 1-cohomology group is just the group of characters of $G$
(i.e., the group of homomorphisms from $G$ to
$\mathbb T$). Nevertheless, there is a class of groups, including Lie
groups, that have no irreducible special representations, but have {\it
reducible} special representations. These representations also give
rise to construct the irreducible representations of current groups in Fock spaces.
In this paper, we confine ourselves to simply connected nilpotent Lie
groups on which there is a nontrivial action of the group
of automorphisms~${\mathbb R\,}^*_+$. In particular, this class contains all
connected, simply connected abelian and metabelian Lie groups, including 
the Heisenberg group and its generalizations.
This choice is justified by the fact that irreducible special
representations of semisimple Lie groups, if they do
exist, in fact come from reducible special representations of
certain nilpotent subgroups. This reduction to nilpotent subgroups is
useful also in other cases. On the other hand, our construction for
nilpotent groups is extremely simple and reduces just to the
construction of semidirect products.

Starting from an arbitrary unitary representation $T$ of a nilpotent
group $N$ of the form described above,
we construct a special (reducible)
representation $\widetilde T$ of this group and then extend it to a
special representation of the semidirect product
$G={\mathbb R\,}^*_+\rightthreetimes N$. After that, we construct the
representations of the corresponding current groups $N^X$ and $G^X$
associated with $\widetilde T$. To describe them, instead of the classical
Fock model, we use the {\it Poisson model} introduced in \cite{VG-1}
(for brevity, we will call it simply the Poisson representation),
which is more convenient for this class of groups.

The case of the Heisenberg group $N$ of order
$2n-1$ is considered separately. We describe the special representation
$\widetilde T$ of the group $P={\mathbb R\,}^*_+\rightthreetimes N$ associated with an
irreducible representation $T$ of the group $N$ and the corresponding
Poisson
representation $U$ of the current group
$P^X$; the representations
$\widetilde T$ and $U$ are also irreducible. We also briefly describe, using
an embedding of $P$
into the semisimple Lie group
$U(n,1)$, an extension of the special representation
$\widetilde T$ of the group $P$ to a representation of the group
$U(n,1)$ and, correspondingly, an extension of the representation of
the current group $P^X$ to a projective representation of the current
group $U(n,1)^X$ (for a detailed description, see \cite{VG-1}).

\smallskip
The authors are grateful to N.~Tsilevich for translating the paper
and useful remarks.

\subsection{Canonical families and the associated special representations
of groups}\label{sect:01}
We say that a unitary representation $T$ of a locally compact group
$G$ in a Hilbert space $H$ is {\it special} if it has a nontrivial
1-cocycle. Recall that a 1-cocycle of $T$ is a
continuous map
$\beta : G\to H$ such that
$\beta (g_1g_2)=T(g_1) \beta (g_2)+ \beta (g_1)$ for any $g_1,g_2\in G$.
Nontriviality means that $\beta $ cannot be written in the form
$T(g)h-h$ with $h\in H$. The construction of special representations
of the group $G$ (if they do exist) is based on the following notion of
a canonical system of representations.

Let $\{T_r,\,0<r\le r_0\}$ be a one-parameter family of unitary
representations of a locally compact group $G$ in Hilbert spaces
$H_r$ and $dm(r)$ be an infinite measure on $[0,r_0]$ such that
$\int_\epsilon  ^{r_0} dm(r)< \infty$  for $\epsilon >0$. In
particular, if $r_0 <  \infty$, we may take $dm(r)=r^{-1}dr$.

\begin{DEF}{}\label{DEF:1-1}
The family
$\{T_r,\,0<r\le r_0\}$ is called a {\it canonical family  of representations
with respect
to the measure $m$} if the spaces
$H_r$ do not contain nonzero vectors invariant under the operators
$T_r(g)$ and there exist unit vectors
$h_r\in H_r$ such that
\begin{equation}{}\label{2-111}
\int_0^{r_0}\|T_r(g)h_r-h_r\|_{H_r}dm(r)< \infty \quad\text{for every}\quad g\in G.
\end{equation}
The family of vectors $\{h_r\}$ will be called almost invariant with
respect to $G$.
\end{DEF}

Note that for the measure $dm(r)=r^{-1}dr$, condition~\eqref{2-111}
follows from a stronger estimate: there exists
$\alpha >0$ such that
\begin{equation}{}\label{2-11}
\|T_r(g)h_r-h_r\|=O(r^{\alpha })\quad\text{for every}\quad g\in G.
\end{equation}

In the definition of a canonical family, we may assume that the spaces
$H_r$ are identified with a fixed Hilbert space $H$ so that the images
of all vectors $h_r$ coincide with a fixed unit vector
$h\in H$. Then condition~\eqref{2-111} takes the following form.

{\it In the space $H$, there is no nonzero vector invariant under the
operators $T_r(g)$ and there is at least one nonzero vector $h$
such that}
\begin{equation}{}\label{2-2}
\int_0^{r_0}\|T_r(g)h-h\|dm(r)< \infty \quad\text{for every}\quad g\in G.
\end{equation}
Such a vector $h$ will be called an almost invariant vector in the space $H$.
Clearly, the set of all almost invariant vectors is a linear subspace
of $H$.

\begin{REM*}{}
It follows from the definition of a canonical family that the identity
representation is a limiting point for
the representations $T_r$ as $r\to 0$ in the Fell topology on
the set of unitary representations of $G$.
\end{REM*}

With every canonical family
$\{T_r,\,0<r<r_0\}$ of representations of a group $G$ in a Hilbert
space $H$ with respect to a measure
$m$ on $[0,r_0]$ we associate a new unitary representation~$\widetilde T$
of this group. It is realized in the direct integral of the Hilbert spaces
$H_r=H$ with
respect to the measure $m$, i.e., in the Hilbert space
$$
\mathcal H=\int_0^{r_0}  H_r dm(r), \quad H_r=H,
$$
of sections $f(r)$ of the fiber bundle over
$[0,r_0]$  with fibers $H_r$, $r\in [0,r_0]$, endowed with the norm
$\|f\|^2=\int_0^{r_0} \|f(r)\|^2 dm(r)$. The operators of this
representation act fiberwise:
$(\widetilde T(g)f)(r)=T_r(g)f(r)$.

\begin{THM}{}\label{THM:31-32} The representation $\widetilde T$ of the
group $G$ associated with a canonical system of representations
$\{T_r : 0<r<r_0\}$ is a special representation with a nontrivial
1-cocycle of the form
\begin{equation}{}\label{31-32}
\beta (g;r)=T_r(g)h_r-h, \quad h_r=h,
\end{equation}
where $h\in H$ is an almost invariant vector in $H$.
\end{THM}

Indeed, it follows from the definition of a canonical family that the vector
$\beta (g)$ with components
$T_r(g)h_r-h_r\in H_r$ lies in the space $\mathcal H$ and, consequently, is a
1-cocycle of the representation
$\widetilde T$ of the group $G$ in the space
$\mathcal H$. The non-triviality follows from the fact that
$H$ has no invariant vectors and the measure $m$ on
$[0,r_0]$ is infinite.

\begin{REM*}{} It is of interest to find out the structure of
representations $T$ that are limiting points
in the Fell topology of the representations $T_r$ of a canonical
system in the case where all $T_r$ are irreducible. If the group $G$
has no special {\it irreducible} representations (it is this case that
we consider in the paper), then the only possible limit is the identity
representation. If there exist irreducible special representations, then
a limiting point can also coincide with one of them. It is this phenomenon
that we observe in the
case where~$T_r$ is a family of complementary series representations
of the semisimple Lie group
$O(n,1)$ or  $U(n,1)$, see \cite{VG}.
\end{REM*}

\subsection{The semidirect products  $G={\mathbb R\,}^*_+\rightthreetimes G_0$
and their special representations}\label{sect:02}
In this section, we introduce a
family of groups having a canonical system of
representations.

Let $G_0$ be an arbitrary locally compact group on which an action of
the group of automorphisms~${\mathbb R\,}^*_+$ is defined:
$r: g_0 \mapsto g_0^r$ for every  $r\in{\mathbb R\,}^*_+$. Then
\begin{enumerate}
\item The group $G_0$ can be realized as a subgroup of the semidirect
product $G={\mathbb R\,}^*_+\rightthreetimes G_0$, i.e., the group of pairs
$(r,g)\in{\mathbb R\,}^*_+ \times G_0$ with the multiplication law
$$
(r_1,g_1)(r_2,g_2)=(r_1r_2, g_1^{r_2}g_2).
$$
\item Every unitary representation $T$ of the group
$G_0$ in a Hilbert space $K$ gives rise to a one-parameter family of
unitary representations $T_r$,  $r\in{\mathbb R\,}^*_+$, of $G_0$ in the same
space $K$, which are given by the formula
$$
T_r(g)=T(g^r).
$$
We say that these representations are {\it conjugate} to the original
representation $T$.
\end{enumerate}

\begin{DEF}{}\label{DEF:41-2} A unitary representation $T$ of the
group $G_0$ in a space $K$ is called {\it summable} with respect to an
infinite measure $m$ on
${\mathbb R\,}^*_+$ with $\int_ \epsilon ^\infty dm(r)< \infty $ for
$\epsilon >0$ if the representations $T_r$ of $G_0$ conjugate to
$T$ form a canonical system with respect to $m$.

The group $G_0$ will be called {\it nice}
if it has at least one summable representation.
\end{DEF}

In what follows, we assume that
$dm(r)=e^{-\frac{u(r)}{2}} d^*r$, where
$d^*r =r^{-1}dr$ is the Haar measure on ${\mathbb R\,}^*_+$ and $u(r)$ is a
smooth positive function on the half-line such that
$\int _ \epsilon ^ \infty  e^{-\frac{u(r)}{2}}d^*r< \infty $ for $\epsilon >0$.

Then the summability condition for a representation $T$ in a space $K$
means
the non-existence of invariant vectors $h\ne 0$ in $K$ and
the existence of a nonzero vector
$h\in K$ such that
\begin{equation}{}\label{2-22}
\int_0^{\infty }\|T_r(g)h-h\|e^{-\frac{u(r)}{2}}d^*r< \infty
\quad\text{for every}\quad g\in G.
\end{equation}

Theorem~\ref{THM:31-32} implies the following result.

\begin{THM}{}\label{THM:31-33} Every summable unitary representation
$T$ of the group $G_0$ in a Hilbert space~$K$ gives rise to a special
representation
$\widetilde T$
of this group in the Hilbert space
$\mathcal K=\int_0^\infty   K_r d^*r$, $K_r=K$, which is defined by the
formula
$(\widetilde T (g)f)(r)=T_r(g) f(r)$.
With every almost invariant vector $h \in K$ is associated
the nontrivial 1-cocycle
$\beta :G_0 \to \mathcal K$,
\begin{equation}{}\label{31-29}
\beta (g,r) =e^{-\frac{u(r)}{2}} (T_r(g)h_r-h_r), \quad h_r=h.
\end{equation}
\end{THM}

\begin{THM}{}\label{THM:31-34} The special unitary representation
$\widetilde T$ of the group $G_0$ can be extended to a representation
of the semidirect product
$G={\mathbb R\,}^*_+\rightthreetimes G_0$ in which the operators corresponding
to elements of the subgroup ${\mathbb R\,}^*_+$ have the form
\begin{equation}{}\label{311-35}
(\widetilde T(r_0)f)(r)=f(rr_0).
\end{equation}
With every almost invariant vector $h \in K$ is associated
the nontrivial 1-cocycle $\beta :G \to \mathcal K$,
\begin{equation}{}\label{31-34}
\beta (g,r)=(\widetilde T(g)f)(r)-f(r),\quad\text{where}\quad
f(r)=e^{-\frac{u(r)}{2}} h_r.
\end{equation}
\end{THM}

\begin{proof}{} The invariance of the Haar measure
$d^*r$ implies that the operators corresponding to elements of the
subgroup ${\mathbb R\,}^*_+$ are unitary. One can easily see that together with
the operators corresponding to elements of the subgroup $G_0$ they
generate a representation of the whole group $G$.

Further, let us verify that
$\beta (g)$ is a nontrivial 1-cocycle on $G$. First of all, for
$g\in G_0$, formula~\eqref{31-34} coincides with
\eqref{31-29}, so that $\beta (g)$ is a nontrivial 1-cocycle on
$G_0$. Hence it suffices to check that
$\|\beta (r_0)\|< \infty $
for $r_0\in{\mathbb R\,}^*_+$. In this case,
$\beta (g,r)=(e^{-\frac{u(rr_0)}{2}}-e^{-\frac{ u(r)}{2}})h_r$, whence
$\|\beta (r_0)\|^2= { \int_0^\infty |e^{-\frac{u(rr_0)}{2}}-
e^{-\frac{u(r)}{2}}|^2 d^*r } < \infty$.
\end{proof}

\smallskip

As we mentioned earlier, there is a subgroup
$K_0 \subset K $ such that $K_0 \setminus 0$
is the set of all almost invariant vectors $h \in K$.
Let us fix a function $u(r)=u_0(r)$.
We can associate to each vector $h \in K_0$
the 1-cocycle $\beta :G\to \mathcal K$ of the form \eqref{31-34}
with $f(r)=e^{-\frac{u_0(r)}{2}}h$.
The subset of all such cocycles forms a subgroup, wchich we denote
$Z_0^1$.

Let $v:K_0$ onto $Z_0^1$ be the natural epimorphism.
If $h,h'\in K_0$, and $h- h'\ne0$, then $h-h'$ is an almost invariant vector.
By Theorem~\ref{THM:31-34} the 1-cocycle $v(h)-v(h')=v(h-h')$ is non-cohomologous to $0$.
This implies:

\begin{COR*}{}
Cocycles in $Z_0^1$ are pairwise non-cohomologous.
The groups $K_0$ and $Z_0^1$ are isomorphic, $K_0\cong Z_0^1$.
\end{COR*}

\begin{THM}{}\label{THM:31-35} Each $1$-cocycle $\beta $ of the form \eqref{31-34}
is cohomologous to a (unique) cocycle
lies in the group
$Z_0^1 \cong K_0 $.
\end{THM}

\begin{proof}{}
Let
$u(0)=u_0(0)=0$, and let $f_0(r)=e^{\frac{-u_0(r)}{2}}h_r$.
Then $f(0)=f_0(0)=h$. Therefore
$$
\int_0^\infty \|f(r)-f_0(r)\| ^2 \,  d^*r   < \infty .
$$
So, $f(r)-f_0(r) \in \mathcal K$, and the 1-cocycle $\beta $
is cohomologous to the 1-cocycle $\beta _0 = v(h)\in  Z_0^1 $.
\end{proof}

\begin{THM}{}\label{THM:31-36} If the  summable representation $T$ of the
group $G_0$ is  irreducible and  the representations $T_r$ of the group $G_0$
conjugate to $T$ are pairwise nonequivalent, then the special representation  $\widetilde T$
of the group $G={\mathbb R\,}^*_+\rightthreetimes G_0$ associated with
$T$ is also irreducible.
\end{THM}

Sometimes the representations $T_r$ are automatically pairwise
nonequivalent; this is the case in our further examples.

\subsection{Special representations of nilpotent Lie groups and their
one-dimensional extensions}
\label{sect:03}

Here we introduce a class of nilpotent Lie groups $N$ having special
representations. Every group in this class can be embedded into the
direct product $G={\mathbb R\,}^*_+\rightthreetimes N$ and is nice in the
sense of Definition~\ref{DEF:41-2}. Note that all special
representations of these groups are reducible, since irreducible
representations of nilpotent groups have no nontrivial 1-cocycles.

\subsubsection{General definitions}
Recall that a Lie algebra $L$ is called nilpotent of class~$n$
if its lower central series
$L=L_1 \supset L_2 \supset \cdots \supset L_n \supset  \cdots$, where
$L_{i+1}=[L,L_i]$, terminates at the $n$th step, i.e.,
$L_n\ne 0$ and $L_{n+1}=0$. We will say that an element
$a\in L$ is of degree
$r(a)=k$ if $a\in L_k \setminus  L_{k+1}$.

The space $L$ of such an algebra can be written as the direct sum of
nonzero vector spaces
\begin{equation}{}\label{11-1}
L= \bigoplus_{k=1}^n R_k,\quad\text{where}\quad  R_k \subset L_k,  \quad
L_k=R_k \oplus L_{k+1}.
\end{equation}
Here {\it every nonzero element of every space $R_k$ has degree
$k$.}

A nilpotent Lie algebra of class~$n$ is called {\it graded} if there
exists a decomposition of the form~\eqref{11-1} such that
\begin{equation*}{}
[R_i,R_j ] \subset R_{i+j}; \quad R_k=0\quad\text{for}\quad k>n.
\end{equation*}
Each such decomposition is called a {\it grading} of the
algebra $L$. In particular, all nilpotent algebras of class~$1$
(abelian algebras) and class~$2$ (metabelian algebras) are graded
algebras.

The following properties are obvious:
\begin{enumerate}
\item Every grading of a graded algebra $L$ is uniquely determined by
a subspace $R_1 \subset L$ such that $L=R_1 \oplus [L,L]$; the
subspaces $R_k$ of this grading can be defined by induction on $k$ via
the formula $R_k=\sum_{i+j=k} [R_j,R_j]$.
\item Every quotient algebra $L/L_k$ of a graded algebra $L$ is also
graded; if
$L= \bigoplus _{i=1}^n R_i$ is a grading of $L$, then
$L/L_k=  \bigoplus_{i=1}^{n-k}R_i/L_k$
is a grading of $L/L_k$.
\end{enumerate}

Every grading~\eqref{11-1} of a graded Lie algebra $L$ of class~$n$
allows one to realize $L$ as the space of sequences
$a=(a_1, \dots ,a_n)$, $a_k\in R_k$, with the multiplication law
\begin{equation}{}\label{11-3}
[a,b]=(0,c_2, \dots ,c_n),\quad\text{where}\quad c_k=\sum_{i+j=k} [a_i,b_j].
\end{equation}

{\it A nilpotent Lie algebra $L$ of class $n$ is a graded algebra if and only if
on $L$ there exists a diagonalizable differentiation operator $D$ with
eigenvalues $1, \dots ,n$.} Namely, the spaces $R_k$ can be
defined as the eigenspaces of this operator.

\begin{PROP}{}\label{PROP:31-2} Every differentiation operator
$D$ with these properties gives rise to a one-dimensional extension of the graded algebra
$L$ to the solvable algebra
$L \oplus \{D\}$ in which the commutation relations on $L$ are
supplemented by the relations
\begin{equation}{}\label{11-5}
[D,D]=0; \quad [D,a]=-[a,D]=Da\quad\text{for}\quad a\in L.
\end{equation}
\end{PROP}

\begin{REM*}{} Examples of nongraded algebras are characteristically
nilpotent algebras, in which all differentiation operators are
nilpotent.
\end{REM*}

A Lie group $N$ is called a nilpotent (graded) group of class~$n$ if its
Lie algebra $L$ is a nilpotent (graded) algebra of class~$n$. It
is known that a simply connected nilpotent group~$N$ is uniquely
determined by its Lie algebra $L$, and there exists a diffeomorphism
$\exp:L\to N$, the exponential map, between the manifolds of this group
and its Lie algebra.

If $N$ is a simply connected graded Lie group of class~$n$ and
\eqref{11-1} is a fixed grading on its Lie algebra $L$, then
one can choose the coordinates of the vector space
$R= \bigoplus _{k=1}^n R_k$  as
coordinates on $N$. We will call them the {\it canonical
coordinates} on $N$. Thus elements
$a\in N$, where $N$ is a one-dimensional graded nilpotent Lie group of
class~$n$, are sequences of the form
$a=(a_1, \dots ,a_n)$ with $a_k\in R_k$.

\begin{THM}{}\label{THM:11-20} In the canonical coordinates, the product of elements of a simply
connected graded nilpotent Lie group $N$ of class~$n$ has the following
form:
\begin{equation}{}\label{11-20}
a \cdot b=(a_1+b_1, a_2+b_2+p_2(a,b)), \dots ,  a_n+b_n+p_n(a,b)),
\end{equation}
where $p_k(a,b)\in R_k$ are polynomials in $a_1, \dots ,a_{k-1}$ and
$b_1, \dots ,b_{k-1}$ satisfying the condition $p_k(a,0)=p_k(0,b)=0$.
In particular, $a \cdot b=a+b+\frac 12 [a,b]$ for $n=2$.
\end{THM}

Given a graded Lie algebra $L$, the action of the differentiation
$D$ on $L$ (i.e., $Da_k=ka_k$ for elements $a_k$ of degree~$k$)
generates an action of the group of automorphisms ${\mathbb R\,}^*_+$ on the
graded group $N= \bigoplus_{k=1}^n R_k$ defined as $a_k \mapsto
r^ka_k$ for $a_k\in R_k$.  Thus $N$ can be extended to a solvable
group $G={\mathbb R\,}^*_+\rightthreetimes N$.

\begin{REM*}{} By a general property of solvable Lie groups, the group
$G={\mathbb R\,}^*_+\rightthreetimes N$ can be realized as a subgroup of the
group of lower triangular matrices of a certain order~$m$. In this
realization, the group
${\mathbb R\,}^*_+$ turns into a one-dimensional subgroup of the group of diagonal
matrices.
\end{REM*}

\subsubsection{Summability of the unitary representations; main theorem}

Recall that G\r{a}rding's space of the unitary representation $T$ of the
Lie group $G$ in the Hilbert space $H$ is the subspace $H^* \subset
H,$ whch consists of all vectors $h\in H$ with property:
 ${\langle T(g)h,h' \rangle} \in C^\infty (G) $ for all $h'\in H.$ According to the theorem by
 Gel'fand-G\r{a}rding, the subspace $H^*$ is dense in $H$.

\begin{THM}{}\label{THM:43-24} Let $N$ be an arbitrary connected and simply connected nilpotent
graded Lie group and let $T$ be an arbitrary unitary representation of
the group $N$ in a Hilbert space $H$.
Suppose $H$ does not contain invariant vectors. 
Then each vector from the G\r{a}rding
space $H^* \subset H$ is almost invariant under $T$,
thus the representation $T$ is summable.
\end{THM}

\begin{proof}{} Let $N\cong \bigoplus_{k=1}^n R_k $, $a_k\in R_k$ be a fixed graduation
of the group $N$ and $a=(a_1, \dots ,a_n)$ -- associated with it
canonical system of coordinates on $N$.
Consider any vector $h\in H^*,$ $\|h\|=1.$
Thus
$$
{\langle T(a)h,h \rangle}=f(a_1, \dots ,a_n)\in C^\infty (N), \quad
f(0)=1\quad \forall \quad  a\in N.
$$
Consequently, because of $T_r(a)=T(a^r)=T(ra_1, \dots ,r^na_n),$ we obtain:
$$
{\langle T_r(a)h,h \rangle}=1+O(r) \quad\text{for small}\quad r.
$$
It means that, $\|T_r(a)h -h\|^2\equiv 2(1-\opn{Re} {\langle
T_r(a)h,h \rangle})=O(r).$ Consequently, the vector $h$ is almost
invariant with respect to the representation $T.$
\end{proof}

\begin{COR*}{}
With every unitary irreducible representation $T$ of the group $N$
is associated the special irreducible representation $\widetilde T$ of the group
$G={\mathbb R\,}^*_+\rightthreetimes N$.
\end{COR*}

\subsection{The special representation of the group ${\mathbb R\,}^*_+\rightthreetimes N$
associated with the regular representation of $N$}
The regular representation of a
locally compact group $G$ is the unitary representation of $G$ in
the Hilbert space $L^2(G,dg)$, where $dg$ is the right-invariant
measure on $G$, by right translations: $(T(g_0 f)(g))=f(gg_0)$.

Let
us describe the structure of the regular representation of a
 connected, simply connected
nilpotent graded Lie group $N$ of class~$n$. Let
$$
N= \bigoplus_{i=1}^n R_i
$$
be a fixed grading of $N$. With this grading, elements of
$N$ can be written as sequences
$a=(a_1, \dots ,a_n)$,
$a_i\in R_i$, and the right-invariant measure on $N$ is
$da=da_1 \cdots da_n$, where $da_i$ is the Lebesgue measure on the
vector space $R_i$.

In our notation, the regular representation $T$ of the group $N$ is
a representation in the Hilbert space $K=L^2(N,da)$ of functions
$f(a)=f(a_1,  \dots,  a_n)$ with the norm $\|f\||^2=\int_N |f(a)|^2
da$. The operators of this representation are defined by the formula
$$
(T(b)f)(a)=f(a \cdot b), \quad  b\in N,
$$
where $a \cdot b$ is given by~\eqref{11-20}.

The operators of the representations $T_r$ conjugate to $T$ have the
form
$$
(T_r(b)f)(a)=f(a \cdot b^r), \quad  b\in N,\quad\text{where}\quad
b^r=(rb_1,r^2b_2, \dots ,r^nb_n).
$$

It is clear that the space $K$ does not contain vectors invariant
under the operators of~$T$,
but,
by Theorem~\ref{THM:43-24}
it contain almost invariant vectors.

Let us describe a family of almost invariant vectors $F_{\mu}$ in the space $K$
of the representation $T$. Let $\mu = (\mu_1, \dots ,\mu_n )$ be an arbitrary sequence
of positive definite quadratic forms $\mu_i: R_i \to {\mathbb R}$.
We assign to 
$\mu$ a function $F_\mu$ on $N:$
\begin{equation}{}\label{1-3-1}
F_{\mu}(a) = \exp (-\mu(a)) = \exp(-\sum_{i=1}^n \mu_i(a_i)),
\end{equation}
where $a=(a_1, \dots ,a_n)$, $a_i \in R_i$.
It is evident that $F_\mu \in K$,
and the linear space spanned by such vectors is
dense in $K$.

\begin{THM}{}\label{THM:6-61} All vectors $F_\mu$ of the form~\eqref{1-3-1}
are almost invariant in the space of the regular representation of
the graded nilpotent group $N$.
\end{THM}

\begin{proof}{} We fix a sequence $\mu = (\mu_1, \dots ,\mu_n )$.
Let $F_k(a) = \exp(- \sum_{i=1}^k \mu_i(a_i))$, $k=1, \dots ,n$.
In particular, $F_n(a)=F_\mu(a)$.
Consider $F_k(a)$ as a function on the group $N/N_{k+1}$, where $N_{i+1}=[N,N_i]$.
Then $F_k(a)$ lies in $H_k$, the space of the regular representation of
$N/N_{k+1}$.

We may assume that the Lebesgue measures
$da_i$ on $R_i$ are normalized so that
$\int_{R_i} e^{-2\mu_i(a_i)} da_i=1$. Then
$\|F_k\|=1$, and the almost invariance condition~\eqref{2-22} for the
vector~$F_k\in H_k$ with respect to the measure
$dm(r)=e^{-\frac {u(r)}{2}}d^*r$ is equivalent to
\begin{equation}{}\label{13-5}
\int_0^1 (1-{\langle T_r(b)F_k,F_k \rangle})^{1/2} d^*r< \infty
\qquad (k=1, \dots ,n).
\end{equation}
We will prove, by induction on the class~$n$, the following estimate:
\begin{equation}{}\label{113-5}
{\langle T_r(b)F_n,F_n \rangle}=1+O(r^2)\quad\text{on}\quad (0,1).
\end{equation}
Obviously, condition~\eqref{13-5} immediately follows from this estimate.

First note that the representations
$T_r$, $r\in{\mathbb R\,}^*_+$, of $N$ conjugate to the regular
representation $T$ have the form
$$
(T_r(b)f)(a)=f(a \cdot b^r), \quad  b\in N,\quad\text{where}\quad
b^r=(rb_1,r^2b_2, \dots ,r^nb_n).
$$
By Theorem~\ref{THM:11-20}, for every fixed $b$ this formula can be
written as
$$
(T_r(b)f)(a)=f(a_1+rb_1, a_2+rq_2(a,r,b),  \cdots ,a_n+rq_n(a,r,b)),
$$
where $q_k(a,r,b)\in R_k$ are polynomials in $a_1, \dots a_{k-1}$ and $r,b_1, \dots ,b_k $.
In what follows, instead of
$q_k(a,r,b)$ we will write $q_k(a)$ or
$q_k(a_1, \dots ,a_{k-1})$.

In the case $n=1$, we have
$$
{\langle T_r(a)F_1,F_1 \rangle}=\int_{R_1}e^{
-\mu_1(a_1+rb_1)-\mu_1(a_1)
} da_1
=e^{-\frac{r^2}{2} \mu( b_1 )}.
$$
Here $\langle \cdot,\cdot \rangle$ is the scalar product on $H_1$.
This immediately implies~\eqref{113-5} for $n=1$.

Let us prove~\eqref{113-5} for an arbitrary $n$ under the
assumption that it is proved for $n-1$. We  use the relation
$$
F_n(a_ 1, \dots ,a_n)=F_{n-1}(a_1, \dots ,a_{n-1})
e^{-\mu_n(a_n)}.
$$
It implies that
$$
{\langle T_r(b)F_n,F_n \rangle}=
\int_N \Bigl( (T_r(b)F_{n-1})(a)\,F_{n-1}(a)\,
e^{-\mu_n(a_n+rq_n(a))
-\mu_n(a_n)
}
\Bigr)da_1 \cdots da_n.
$$
Here $\langle \cdot,\cdot \rangle$ is the scalar product on $H_{n}$.
Integrating with respect to $a_n$ and observing that only
the last factor in the integrand depends on $a_`n$, we obtain
\begin{equation}{}\label{117-6}
{\langle T_r(b)F_n,F_n \rangle}=
\int_N {\langle T_r(b)F_{n-1}(a),F_{n-1}(a) \rangle})
e^{-\frac{r^2}{2}\mu_n( q_n(a))
}  da_1 \cdots da_{n-1}.
\end{equation}
Therefore,
\begin{equation}{}\label{17-6}
{\langle T_r(b)F_n,F_n \rangle}={\langle T_r(b)F_{n-1},F_{n-1} \rangle}-
r^2\int (Q(a){\langle T_r(b)F_{n-1}(a),F_{n-1}(a) \rangle}) da_1
\cdots da_{n-1},
\end{equation}
where $Q(a)$ is a bounded function of
$r$ and $a_1, \dots a_{n-1}$. Since
$Q(a)$ is bounded, the function
$I(r)=\int (Q(a){\langle T_r(b)F_{n-1},F_{n-1} \rangle}) da_1 \cdots da_{n-1}$
is bounded on $[0,1]$. Hence, in view of \eqref{17-6}, estimate~\eqref{13-5}
for $F_{n-1}$ implies the same estimate for $F_n$.
\end{proof}

\begin{COR*}{}
The special representation $\widetilde T$
of the group $G={\mathbb R\,}^*_+\rightthreetimes N$, associated
with regular representation $T$ of the group $N$, has a family of
nontrivial $1$-cocycles of type
$$
\beta_\mu (g,r)=(\widetilde T(g)f_\mu)(r)-f_\mu(r),\quad\text{where}\quad
f_\mu(r)=e^{-\frac r2} F_\mu 
$$
parametrized by sequences $\mu = (\mu_1, \dots ,\mu_n )$
of positive definite quadratic forms $\mu_i: R_i \to {\mathbb R}$.
By Theorem~\ref{THM:31-34} (corollary)
these cocycles $\beta_\mu$ are pairwise non-cohomologous.
\end{COR*}

\subsection{The case of the Heisenberg group}\label{sect:04}
Here we describe the special representation of the Heisenberg group $N$ of
order $2n-1$ associated with an irreducible representation of $N$ and
its extension to a representation of the group
$P={\mathbb R\,}^*_+\rightthreetimes N$.

The group $N$ is a nilpotent group of class~$2$, so that we can apply
the above definitions and constructions. We realize $N$ as the set of
pairs  $(it,z)$, $t\in {\mathbb R\,}$, $z\in{\mathbb C\,}^{n-1}$, or, which is more
convenient, as the set of pairs
$$
(\zeta ,z),\quad\text{where}\quad  \zeta = it-\frac 12 |z|^2, \quad
t\in {\mathbb R\,}, \quad  z\in{\mathbb C\,}^{n-1},
$$
with the multiplication law
$$
(\zeta_1 ,z_1)(\zeta_2 ,z_2)=(\zeta_1+\zeta_2-z_1z^*_2 ,z_1+z_2).
$$
The group of automorphisms ${\mathbb R\,}^*_+$ acts on $N$ according to the formula
$(\zeta ,z) \mapsto (r^2\zeta ,rz)$, so that we have an embedding of $N$ into
the semidirect product
$P={\mathbb R\,}^*_+\rightthreetimes N$.

Up to conjugacy of representations with respect to the action
of the group ${\mathbb R\,}^*_+$, there are two irreducible infinite-dimensional
unitary representations of the group $N$. The first representation $T$
can be realized in the Hilbert space
$K$ of entire analytic functions
$f(z)$ on ${\mathbb C\,}^{n-1}$ with the norm
\begin{equation}{}\label{58-44}
\|f\|^2
=\int_{{\mathbb C\,}^{n-1}}|f(z)|^2\,e^{-zz^*}\,d\mu(z),
\quad zz^*=\sum z_i\overline z_i; \quad
\end{equation}
its operators have the form
\begin{equation}{}\label{14-51}
(T(\zeta _0,z_0)f)(z)
=e^{\zeta _0-zz^*_0}\,f(z+z_0).
\end{equation}
The second representation can be realized in the space of entire
antianalytic functions. For definiteness, we confine ourselves to
considering the first representation $T$.

The special representation $\widetilde T$ of the group $N$ associated with
$T$ is realized in the space
$\mathcal K=\int_0^\infty K_r d^*r$ and is given by the formula
\begin{equation}{}\label{14-5}
(\widetilde T(\zeta _0,z_0)f)(r,z)=e^{r^2\zeta _0-rzz^*_0}\,f(r,z+rz_0).
\end{equation}
It can be extended to a special unitary
representation
of the group $P={\mathbb R\,}^*_+\rightthreetimes N$, where the operators
corresponding to elements of the subgroup ${\mathbb R\,}^*_+$
are given by the formula
\begin{equation}{}\label{14-52}
(\widetilde T(r_0)f)(r,z)=f(rr_0,z).
\end{equation}
{\it The representation $\widetilde T$ of the group $P$ is irreducible and
has a nontrivial 1-cocycle of the form}
\begin{equation}{}\label{14-13}
b(g;r,z)=(\widetilde T (g)f)(r,z)-f(r,z),\quad\text{where}\quad f(r,z)=e^{-r^2}.
\end{equation}

There exists an embedding of $P$ into the semisimple Lie group
$U(n,1)$ such that every element
$g\in U(n,1)$ can be uniquely written as a product
$g=pu$, $p\in P$,  $u\in U$, where $U$ is the maximal compact subgroup
of $U(n,1)$. Using this embedding, we briefly describe the extension
of the special representation $\widetilde T$  of the group $P$ to a special
unitary representation of the group
$U(n,1)$ (for more details, see
 \cite{VG-1}).

Consider the homogeneous space
$L=U \backslash  U(n,1) $ of the group $U(n,1)$
realized as the space of pairs
$v=(a,b)\in{\mathbb C\,} \oplus {\mathbb C\,}^{n-1}$, where $a+\overline a +bb^*<0$. The
subgroup $U$ is the stationary subgroup of the point
$v_0=(-1,0)$, so that every element
$v\in L$ can be written in the form $v=pv_0$ with $p\in P$.

With every point
$v=(a,b)\in L$ we associate the following function $f_v$
on ${\mathbb R\,}^*_+ \times {\mathbb C\,}^{n-1}$:
\begin{equation}{}\label{4-11}
f_v(r,z)=e^{r^2a+(z,b)},\quad\text{where}\quad (z,b)=\sum z_ib_i.
\end{equation}
For every fixed $r$, these functions lie in the space $K$ and form a
total subset in $K$;
their differences
$f_{v_1}-f_{v_2}$ lie in the space
$\mathcal K=\int_0^\infty K_r d^*r$ and form a total subset
$M$ in $\mathcal K$.

On the set $M$, the original formulas for the operators
$\widetilde T(g)$ of the special representation of $P$ are equivalent to the
following ones:
\begin{equation}{}\label{4-12}
\widetilde T(g)(f_{v_1}-f_{v_2})=f_{gv_1}-f_{gv_2}.
\end{equation}
\begin{DEF}{} The extension of the representation
$\widetilde T$ of the subgroup $P$ to the group $U(n,1)$ is defined on the
total set $M$ by the same formula~\eqref{4-12}.
\end{DEF}
The nontrivial 1-cocycle of this representation has the form
$$
b(g)=f_{gv_0}- f_{v_0}.
$$

\subsection{The Poisson representations of current groups
$N^X$, where $N$ is a graded nilpotent Lie group}\label{sect:05}
We begin with brief definitions of the countable tensor product of
Hilbert spaces and quasi-Poisson measures.

The {\it countable tensor product
$\otimes _{n=1} ^\infty (K_n,h_n)$ of punctured Hilbert spaces}
$(K_n,h_n)$, where $h_n\in K_n$ are fixed
unit vectors in Hilbert spaces $K_n$, is the Hilbert space obtained by the norm
completion of the inductive limit of the finite tensor products
$H_k=\otimes _{n=1} ^k K_n$ with respect to the embeddings
$H_k\to H_k \otimes h_{k+1} \subset H_{k+1}$.

{\it The Poisson and quasi-Poisson measures.} Every pair
$(Y,\mu)$, where $Y$ is a standard Borel space (the state space) and
$\mu$ is an infinite $\sigma$-finite continuous
measure on $Y$ (the intensity measure) gives rise to the Poisson pair
$(\mathcal E(Y),\nu)$, where
$\mathcal E(Y) $ is the space of all configurations on $Y$
(finite or countable collections of points from $Y$ containing
every point with at most finite multiplicity) and
$\nu$ is the Poisson measure on
$\mathcal E(Y)$ with intensity measure $\mu$;
for more details, see \cite{KG} and \cite{VG-1}.

A {\it quasi-Poisson measure} associated with the Poisson measure
$\nu$ is an infinite (sigma-finite) measure $\sigma $ on $\mathcal E(Y)$
that is equivalent with respect to $\nu$;
for more details, see \cite{VG-1}.

In what follows, we assume that
$(Y,\mu) =({\mathbb R\,}^*_+, e^{-u(r)}d^*r) \times (X,m)$, where $(X,m)$ is a
space with a probability measure $m$ and
$u(r)$ is a fixed positive function on the half-line such that
$\int_\epsilon ^\infty e^{-\frac {u(r)}{2}} d^*r < \infty $ for $\epsilon >0$.
Thus $\mathcal E(Y)$ is the space of countable sequences
$\omega =\{(r_i,x_i)\}$, where
$r_i\in{\mathbb R\,}^*_+$, $X_i\in X$, and $\nu$ is the Poisson measure on
$\mathcal E(Y)$ with intensity measure
$e^{-u(r)} d^*r dm(x)$. In what follows, we assume that $\sigma$ is the
quasi-Poisson measure $\sigma$ associated with $\nu$ that is
given by the formula
$$
d \sigma (\omega )=\pi(\omega ) d\nu(\omega ),\quad\text{where}\quad
\pi(\omega )=e^{-\sum _{(r,x)\in \omega } u(r)}.
$$
The measure $\sigma $ is uniquely determined by its characteristic
functional, which has the following form:
\begin{equation}{}\label{41-41}
\int_{\mathcal E(Y)} \exp \Bigl( -\sum_{y\in \omega }f(y)\Bigr)d \sigma (\omega )
= \exp \Bigl( \int_{{\mathbb R\,}^*_+ \times X} (e^{-f(r,x)}-e^{-u(r)} d^*r d m(x)\Bigr).
\end{equation}

Now we describe the quasi-Poisson
representation of the
current group $N^X$ associated with a special representation $T$ of
the group $N$. Let $K$ be the space of the representation~$T$ and
$h\in K$ be an almost invariant vector of $T$. With
every configuration $\omega \in \mathcal E(Y)$ we associate the countable
tensor product of punctured
Hilbert spaces $(K_y,h_y)=(K,h)$:
$$
K^ \otimes _ \omega =  \bigotimes _{(r,x)\in \omega } K_{r,x}, \quad
K_{r,x}=K.
$$
The representation $U$ of the group $N^X$ associated with the
representation $T$ of the group~$N$ can be realized in the Hilbert space
$$
\opn{QPS}(\mathcal E(Y), \sigma ,K^\otimes _{\omega })
$$
of sections $f(\omega )$ of the fiber bundle $\mathcal E(Y)$ with fibers
$K^\otimes _{\omega }$,
$\omega \in\mathcal E(Y) $, endowed with the norm
$$
\|f\|^2 =\int_{\mathcal E(Y)} \|f(\omega )\|^2 d \sigma (\omega ).
$$
The operators of this representation
are defined on the set of vectors of the
form $F_v (\omega )= \otimes _{(r,x)\in \omega } v(r,x)$ (which is total in
$\opn{QPS}(\mathcal E(Y), \sigma ,K^\otimes _{\omega })$) by the
following formula:
$$
(U(g(\cdot ))F_v)(\omega )= \bigotimes _{(r,x)\in \omega } T_r(g(x)) v(r,x),
$$
where $T_r$ are the representations of $N$ conjugate to the special
representation $T$.

The representation $U$ of the group $N^X$ is obviously reducible.

\subsection{The extension of the representation $U$ of the group $N^X$
to a representation of the group $G^X$, where
$G={\mathbb R\,}^*_+\rightthreetimes N$}\label{sect:06}

Now we describe an extension of the representation~$U$
of the group $N^X$ in the quasi-Poisson space
$\opn{QPS}(\mathcal E(Y), \sigma ,K^\otimes _{\omega })$ constructed
above to a representation of
the group
$G^X=({\mathbb R\,}^*_+)^X\rightthreetimes N^X$, where, by definition,
$({\mathbb R\,}^*_+)^X$ is the group of measurable maps
$\widetilde r: X\to{\mathbb R\,}^*_+$ such that
$$
\int_X |\log  r(x)| dm(x)< \infty .
$$
There is a natural action of the group $({\mathbb R\,}^*_+)^X$ by transformations
$$
(r,x)\mapsto \widetilde r(r,x)=(r(x)r,x)
$$
of the space $Y={\mathbb R\,}^*_+ \times X $
and, consequently, by transformations
$\omega \mapsto \widetilde r \omega $ of the space of configurations
$\mathcal E(Y)$. It is known (see \cite{VG-1}) that
{\it the quasi-Poisson measure $\sigma $ is ergodic and projectively
invariant under the group of transformations
$({\mathbb R\,}^*_+)^X$:}
$$
d \sigma (\widetilde r \omega )=e^{\int_X\log r(x) dm(x)} d \sigma (\omega ).
$$

With the elements of the group $({\mathbb R\,}^*_+)^X$ we associate the  operators on
$\opn{QPS}(\mathcal E(Y), \sigma ,K^\otimes _{\omega })$ given by the formula
\begin{equation}{}\label{48-31}
(U(\widetilde r)f)(\omega )=e^{\frac 12\int_X\log r(x) dm(x)} f(\widetilde r \omega ).
\end{equation}

\begin{THM}{}\label{THM:45-11} The operators $U(\widetilde r)$ are unitary;
together with the operators corresponding to the
elements of the subgroup
$N^X$, they generate a unitary representation of the whole group $G^X$.
\end{THM}

Indeed, the operators  $U(\widetilde r)$ are unitary by the projective
invariance of the measure $\sigma$. The second assertion
immediately follows from the definition of these
operators.

Note that the operators $U(g)$ for the whole group $G^X$ can be
defined by a single formula:
\begin{equation}{}\label{7-17}
(U(g)F_v)(\omega )=e^{\frac 12\int_X\log\widetilde r(x) dm(x)} F_{\widetilde T(g)v},
\end{equation}
where $r(\cdot ) $ is the image of $g(\cdot ) $ under the homomorphism
$G^X  \to({\mathbb R\,}^*_+)^X$.

\begin{THM}{}\label{THM:48-32}
If the summable representation $T$ of the group $N$ satisfy
to the conditions of Theorem~\ref{THM:31-36},
where $G_0=N,$ then the representation $U$ of the group
$G^X={\mathbb R\,}^*_+\rightthreetimes N^X$ in the space
$\opn{QPS}(\mathcal E(Y ), \sigma ;K _{\omega }^ \otimes )$
associated with $T$ is irreducible.
\end{THM}

\begin{proof}
Let us prove that every bounded operator $A$ on
$\opn{QPS}(\mathcal E(Y ), \sigma ;K _{\omega }^ \otimes )$ that commutes
with the operators $U(g),$ $g\in G^X,$ is a scalar operator.

It follows from Theorem~\ref{THM:31-36} that the representations
$U_{\omega}(g)$ of the group $N^X$ in the spaces $K^\otimes_{\omega}$
are (up to a subset of zero measure in the space $\mathcal E(Y )$)
irreducible and pairwise nonequivalent. Hence, every operator $A$ on
$\opn{QPS}(\mathcal E(Y ), \sigma ;K _{\omega }^ \otimes )$ that commutes
with the operators
$U(g),$ $g\in N^X,$ is scalar on every $K^\otimes_{\omega}$.
i.e., is the multiplication by a function $f(\omega)$.
If $A$ commutes also with the operators $U(g)$, $g\in ({\Bbb R}^*_+)^X,$
then $f(\omega)$ is constant on the orbits of this group
un the space ${\rm E}(Y).$ 
Since the measure $\sigma$ is ergodic with respect to $({\Bbb R}^*_+)^X,$
it follows that $A$ is a scalar operator.
\end{proof}

\subsection{The Poisson representation
of the current group $P^X=({\mathbb R\,}^*_+)^X\rightthreetimes N^X$,
where $N$ is the Heisenberg group of order $2n-1$, and its extension
to a representation of the current group $U(n,1)^X$}\label{sect:07}

Let $T$ be the irreducible representation of the Heisenberg group $N$
in the Hilbert space $K$ described in Section~\ref{sect:04} and
$\sigma $ be the quasi-Poisson measure on
$\mathcal E(Y)=\mathcal E({\mathbb R\,}^*_+ \times X)$ with characteristic
functional~\eqref{41-41}, where $u(r)-2r^2$.

We will describe the representation $U$ of the current group
$P^X$ in the quasi-Poisson space
$\opn{QPS}(\mathcal E(Y), \sigma ,K^\otimes _{\omega })$ associated with $T$.

Let $\widetilde T$ be the special representation of the group $P$ associated
with $T$, see Section~\ref{sect:04}. Consider the following total subset of
vectors in
$\opn{QPS}(\mathcal E(Y), \sigma ,K^\otimes _{\omega })$:
\begin{equation}{}\label{24-56}
F_p (\omega )=F_{\widetilde T(p) f}= \bigotimes _{(r,x)\in \omega } (\widetilde T(p(x))f)(r,x;z),
\quad\text{where}\quad f(r,x;z)=e^{- r^2}
\end{equation}
and $p$ ranges over the group $P^X$. By the general formula~\eqref{7-17},
the operators $U(p)$ for elements of  the group $P^X$ act on
this set according to the formula
\begin{equation}{}\label{39}
U(p_0)F_p=e^{\frac 12\int_X\log r_0(x) dm(x)} F_{p_0p},
\end{equation}
where $r_0(\cdot ) $ is the image of $p_0(\cdot ) $
under the homomorphism $P^X  \to{\mathbb R\,}^*_+$.

It follows from the irreducibility of $T$ that the representation
$U$ of the current group $P^X$ is also irreducible.

According to Section~\ref{sect:04}, the special representation
$\widetilde T$ of the group $P$ can be extended to a representation of the
group $U(n,1)$. Using this extension, we will describe an extension of
the representation of $P^X$ to a representation of
$U(n,1)^X$.

In the space $\opn{QPS}(\mathcal E(Y), \sigma ,K^\otimes _{\omega })$
we consider the set $M$ of vectors of the form
\begin{equation}{}\label{57}
F_g(\omega )= \bigotimes _{(r,x)\in \omega } (\widetilde T(g(x))f)(r,x;z), \quad
g\in U(n,1)^X.
\end{equation}

Note that $F_{gu}=F_g$ for every $u\in U^X$. Thus $M$
coincides with the total set~\eqref{24-56} of vectors $F_p$, $p\in P^X$.

We define the operators $U(g_0)$ for
$g_0\in U(n,1)^X$ on the set $M$ by the following formula:
$$
U(g_0) F_g=e^{\int_X l \lambda (g_0(x),g(x))dm(x)} F_{g_0g},
$$
where
$$
\lambda (g_0,g)=-\frac 12 \|b (g_0)\|^2 -\opn{Re}  {\langle b(g_0),f \rangle}-
{\langle \widetilde T(g_0) b(g), \widetilde T(g_0) f \rangle }+
{\langle  b(g),  f \rangle }.
$$

{\it The restrictions to $P^X$ of the operators
$U(g)$ thus defined coincide on $M$ with the operators
$U(p)$ of the representation of the group $P^X$ defined above.}

\begin{THM}{}\label{THM:72} The operators $U(g)$, $g\in U(n,1)^X$,
preserve the inner products on $M$ and are related by the formula
$$
U( g_1 g_2) F_g=e^{i\int_X \rho ( g_1(x), g_2(x))dm(x)}\,
U( g_1)\,U( g_2) F_g\quad\text{for any}
\quad  g_1, g_2, g\in U(n,1)^X,
$$
where
$$
\rho ( g_1, g_2)=-\opn{Im}({\langle  b( g_2),
b( g_1^{-1} \rangle}))+
\opn{Im} {\langle  b( g_1 g_2) - b( g_1)- b( g_2), f  \rangle}.
$$
Therefore, they induce an extension of the original representation of
the group $P^X$ to a unitary projective representation of the group
$U(n,1)^X$.
\end{THM}

\bigskip

\renewcommand\section\subsection

\end{document}